\documentclass[11pt,reqno]{amsart}

\usepackage{amsthm}
\usepackage{amsfonts}
\usepackage{mathrsfs}
\usepackage{amsmath}
\usepackage{amssymb, amsmath}
\usepackage{color}
\usepackage{graphicx}
\usepackage{dsfont}

\newtheorem{Theorem}{Theorem}
\newtheorem{Proposition}{Proposition}[section]
\newtheorem{Lemma}{Lemma}[section]

\newtheorem{Remark}{Remark}[section]

\numberwithin{equation}{section}

\def\div{\mathop{\rm div}\nolimits}

\def\eqdefa{\buildrel\hbox{\footnotesize def}\over =}

\newcommand{\beq}{\begin{equation}}
\newcommand{\eeq}{\end{equation}}
\newcommand{\ben}{\begin{eqnarray}}
\newcommand{\een}{\end{eqnarray}}
\newcommand{\beno}{\begin{eqnarray*}}
\newcommand{\eeno}{\end{eqnarray*}}
\newcommand{\bali}{\begin{aligned}}
\newcommand{\eali}{\end{aligned}}

\newcommand{\ve}{\epsilon}

\newcommand{\na}{\nabla}

\newcommand{\ud}{\mathrm{d}}

\newcommand{\ii}{\mathbf{i}}
\newcommand{\jj}{\mathbf{j}}

\newcommand{\kk}{\mathbf{k}}

\newcommand{\RR}{\mathbb{R}}

\newcommand{\CR}{\mathcal{R}}

\newcommand{\CU}{\mathcal{U}}

\newcommand{\BS}{{\mathbb{S}^2}}

\def\({\left (}
\def\){\right )}

\def\p{\partial}
\def\O{\Omega}

\begin{document}

\title{ The Oseen-Frank limit of Onsager's molecular theory for liquid crystals}

\author{Yuning Liu }
\address{Joint Mathematics Institute,   New York University Shanghai, 200122, Shanghai, P.R.China}
\email{yl67@nyu.edu}

\author{Wei Wang}
\address{Department of Mathematics, Zhejiang University, 310027, Hangzhou, P. R. China}
\email{wangw07@zju.edu.cn}


\vspace{-0.3in}
\begin{abstract}
We study the relationship between Onsager's molecular theory and the Oseen-Frank theory  for nematic liquid crystals. Under the molecular setting, we consider the free energy that includes the effects of nonlocal molecular interactions. By imposing  the strong anchoring boundary condition on the second moment of the number density  function,  we prove the existence of global minimizers for the free energy. Moreover,  when the re-scaled interaction distance tends to zero, the  corresponding global minimizers will converge to an uniaxial distribution whose orientation is described by a minimizer of Oseen-Frank energy.
\end{abstract}

\maketitle

\section{Introduction}
The liquid crystal state is a distinct phase of matter   that is between those of ordinary liquid and solid crystal.
They may flow as a liquid while the molecules are oriented in a crystal-like way. A classification of liquid crystals based on their structural properties was first proposed by G. Friedel in 1922 and they are generally divided into three main classes: in the {\it nematic phase}, the molecules tend to have the same alignment but their positions are not correlated. In the {\it cholesteric phase}, the molecules tend to have the same alignment which varies regularly through the medium with a periodicity distance. In  the {\it smectic  phase}, the molecules are arranged in layers and exhibit some correlations in their positions in addition to the orientational ordering.

In this work, we shall restrict ourselves to   the nematic case. Since  the long planar molecules  usually involved are symbolized by ellipses, they  can be characterized by long-range orientational order: the long axis of the molecules tend to align along a preferred direction.
There are mainly three kinds of continuum  theories, which use different  {\it order parameters} in physics, to capture such anisotropic behavior of liquid crystals.
Among them, the most intuitive one is the {\it vector  theory}, which uses a unit-vector field $n(x)$ to represent the locally preferred  direction that the liquid crystal molecules self-orient themselves
near $x$. On this direction, the most well-known model is the Oseen-Frank  theory based on curvature elasticity theory, in which  the distortion energy of the liquid crystals is  characterized by
the following  Oseen-Frank energy:
\begin{equation}\label{eq:oseen}
  E_{OF}[n]=\tfrac {k_1} 2(\na\cdot n)^2+\tfrac {k_2} 2( n{\cdot}(\na\wedge n))^2
+\tfrac {k_3} 2|n{\wedge}(\na\wedge n)|^2
+\tfrac{k_2+k_4}2\big(\textrm{tr}(\na n)^2-(\na\cdot n)^2\big)
\end{equation}
where $k_1, k_2, k_3, k_4$ are elasticity constants and $\wedge$ denotes the wedge product for two vectors in $\RR^3$. The first three terms in \eqref{eq:oseen} correspond to the three typical pure deformations:
splay, twist and bend while  the last component is actually a null lagrangian due to Ericksen   \cite{Ericksen1976}. Various analytic results related to the global minimizers of \eqref{eq:oseen} under Dirichlet boundary condition is investigated in \cite{HardtKinderlehrerLin1986}. In the simplest setting: $k_1=k_2=k_3=k_F, k_4=0$,  Oseen-Frank energy \eqref{eq:oseen}
  reduces to the Dirichlet energy
\begin{equation}\label{eq:harmonic}
E_{OF}[n]=\tfrac {k_F} 2|\na n|^2.
\end{equation}
Minimizing \eqref{eq:harmonic} among  mappings from $\O$ into $\BS$ under certain boundary conditions   leads to  harmonic maps into $\BS$, which  are widely studied in the past few decades, see for example \cite{LinWang2008}.
For the purpose of describing the hydrodynamics of liquid crytals,  Ericksen and Leslie
formulated the full system in \cite{EricksenARMA1962Hydrostatics,LeslieARMA1968}. This system also attracts the interests of  analysts, especially those working on geometric analysis, for its relationship with harmonic map heat flow. See \cite{LinWang2014} for a  review  of   recent progresses on the  mathematics of Ericksen-Leslie system.

The second theory, which will be investigated in this work, is the {\it molecular theory}. This is a microscopic theory which uses a number density function
$f(x,m)$ to characterize, at each point $x$, the number density of molecules whose orientations are parallel to the direction $m\in\BS$. It  was first presented by Onsager in \cite{onsager1949} to model the isotropic-nematic phase transition and later developed by Maier and Saupe in a series of paper, see for example \cite{MaierSaupe1958}. In \cite{doi1988theory}, this theory is developed by Doi et al. for the sake of studying the  hydrodynamics of liquid crystals.  In Onsager's theory, each spatial position $x$ of   material occupying $\Omega\subset\mathbb{R}^d$ is associated with a   number density function
 \begin{equation*}
   f(x,m):\O\times\BS\mapsto \overline{\mathbb{R}^+}
 \end{equation*}
 which indicates the fraction of molecules per unit solid angle having various orientations. Onsager proposed a mean-field model to
describe isotropic-nematic phase transition for liquid crystals. His expression for the  free energy, at each fixed $x\in\O$ takes the following form:
\begin{equation}\label{Intro-energy}
  A[f](x)=\int_{\BS}\(f(x, m)\log{f(x, m)}+f( x,m)\CU[f](x,  m)\)~\ud m,
\end{equation}
where $\CU[f]( x,m)$ is the mean-field interaction potential defined by
\begin{equation}\label{potential-ms}
  \CU[f](x, m)=\int_{\BS}B(m,m')f(x,m')\ud m'~\text{with}~B(m,m')=\alpha |m\wedge m'|^2.
\end{equation}
Here  $B(m,m')$ is the interaction potential between two molecules with orientations $m$ and $m'$ respectively and $\alpha$  is a parameter
that measures the intensity of the potential.
In Onsager's original treatment \cite{onsager1949}, $B(m,m')$ is chosen to be
\begin{align}\label{potential-on}
  B(m,m')=\alpha |m\wedge m'|,
\end{align}
  calculated from the excluded-volume potential for hard rods. The form \eqref{potential-ms} was introduced later by
Maier and Saupe and is widely employed for it shares qualitatively the same features as Onsager's original one \eqref{potential-on} at the same time  easier to handle analytically. Especially, energy \eqref{Intro-energy} with \eqref{potential-ms} is closely  related  to the variant   second moment of $f$, which is usually called $Q$-tensor and defined by
\begin{equation}\label{eq:1.05}
  Q[f](x)=\int_{\BS}\(m\otimes m-\tfrac 13\mathbb{I}_3\)f(x,m)\ud m.
\end{equation}
More precisely,
\begin{equation}\label{eq:1.02}
  A[f]=\int_\BS f\log f\ud m+\alpha\(\tfrac23-|Q[f]|^2\).
\end{equation}
In order to  characterize the distortion elasticity energy,  each family of number density function
 $f(x,m)$
 is associated with a nonlocal free energy, proposed in \cite{EZ, yu2007kinetic}
\begin{equation} \label{eq:free energy-inhom-intro}
{A}_\ve[f]=\int_\O\int_{\BS}\(f( x, m)\log{f( x, m)}
+f( x, m)\CU_\ve [f](   x,m)\)\ud m\ud x.
\end{equation}
Here  $\CU_\ve$ is a spatial nonhomogeneous mean-field  potential, chosen to be
\begin{equation}\label{inhomopot}
  \CU_\ve[f](x,m)=\int_\Omega\int_\BS{B}( x, m;  x', m')f( x', m')\ud x'\ud m',
\end{equation}
where ${B}(x, m; x', m')$ is the interaction kernel between two molecules
with different configuration $(x,m)$ and $(x',m')$ respectively. In this paper, we will follow \cite{EZ, yu2007kinetic} and choose $B$ like
\begin{align}\label{eq:On-interaction}
B( x, m; x', m')=\alpha| m\wedge m'|^2g_\ve(x-x')
\end{align}
with
\begin{equation*}
  g_\ve(x):=\tfrac 1{\sqrt{\ve}^3}g(\tfrac x{\sqrt{\ve}})
\end{equation*}
where  the parameter $\sqrt\ve$ denotes the re-scaled length of the molecule as well as the typical molecular interaction distance.
Throughout this work,   the interaction kernel $g$ will be a non-negative, smooth, radial function with exponential decay and satisfies
$\int_{\RR^d}g(x) =1$.

 The third continuum theory for nematic liquid crystal   is called Landau-de Gennes theory, or {\it $Q$-tensor theory}. In this framework, the order parameter is   a $3\times 3$ traceless symmetric
matrix-valued function $Q(x)$ characterizing the orientation of molecules near $x$, see \cite{DeGennesProst1995} for instance. As a phenomenological theory, it was derived   based on the thermodynamical consideration of the Gibbs free energy (see \cite{Landau1958}) of the system  and it  gives a phenomenological description of the nematic-isotropic phase transition of liquid crystals. See \cite{MZ} for a mathematical approach.

These three theories all have made great successes on describing either the statics or the dynamics of liquid crystals. In the meantime,    efforts have been made to establish  the relationships between these three theories.
In the   static case, for instance, the attempt  on determining the  elastic constants in Oseen-Frank  energy \eqref{eq:oseen} from molecular theories can be found in \cite{SomozaTarazona1989}.
In \cite{maiersaupe}, a singular bulk energy for $Q$-tensor model is derived from the entropy term in Onsager's molecular theory.
 A systematic way of  deriving macroscopic models, namely, $Q$-tensor models and vector  models,
from Onsager's molecular theory is developed in \cite{HLWZZ}.

For the hydrodynamics, Kuzuu-Doi \cite{KD} formally derive the Ericksen-Leslie equation from the homogenous Doi-Onsager
equation and determine the Leslie coefficients by taking Deborah number to be $0$. For the sake of recovering Ericksen stress tensor from microscopic theory, in  \cite{EZ} and \cite{wang2002kinetic}, a kinetic model for nonhomogeneous
liquid crystalline polymers is presented and the full Ericksen-Leslie equation is derived.
In \cite{MR3366748}, the authors rigorously proved that,  when Deborah number tends to $0$, the  solution of Doi-Onsager equations will converge
to the smooth solution of Ericksen-Leslie equations. On the other hand, in \cite{FCL}, various dynamical $Q$-tensor model is derived from Doi's kinetic theory by using different moment closure methods. See also
 \cite{HLWZZ} for  a dynamical $Q$-tensor model derived from kinetic theory which satisfies energy dissipation law.

Concerning the   connection between $Q$-tensor theory and vector theory,
one can consult   \cite{bauman2012analysis,golovaty2014minimizers,MZ} for the static case and \cite{wang2015rigorous} for the dynamics. In these works,  the asymptotic behavior of small
elasticity coefficients  of Landau-de Gennes model is rigorously analyzed and the Oseen-Frank  model or the Ericksen-Leslie model is recovered in the limit. These results shows that the $Q$-tensor theory and vector theory agree well away from the singularities.

For the sake of understanding  the liquid crystal defects  predicted by Oseen-Frank  model or Ericksen-Leslie system,  the weak solution provides a suitable framework. However,
the connections between weak solutions  of the molecular models and that of the vector models are not fully explored and this article is intended to contribute to this direction.
We shall study the asymptotic behaviors of  the minimizers of \eqref{eq:free energy-inhom-intro}-\eqref{eq:On-interaction} under strong anchoring boundary condition. Our main result  is,
when the parameter $\ve$   tends to $0$, the global minimizers   will converge to uniaxial distributions  that are parameterized  by     harmonic maps.

Through this work, $\O$ will be a simply-connected, bounded domain of class $C^0$ in $\mathbb{R}^d$ with $d=2,3$.
 We shall consider the number density function in the following function space
\begin{equation}\label{eq:1.09}
  \mathscr{H}(\O)=\{f\in L^1(\BS\times\Omega),f(x,m)\geq 0,\|f(\cdot,x)\|_{L^1(\BS)}=1,~a.e.~ x\in\Omega\}.
\end{equation}
For each fixed $\ve>0$, we   consider  the minimizing problem
  \begin{equation}\label{globalmini2}
      \inf_{f\in\mathscr{A}}  A_\ve[f]
  \end{equation}
  in the admissible space
  \begin{equation}\label{admissible}
    \mathscr{A}:=\left\{f\in \mathscr{H}(\O)~|~Q[f](x) = Q[h_{n_b}](x)~\text{in}~\O^\delta \right\}.
  \end{equation}
Recall that $Q[f]$ is the variant second moment  of $f$ defined by \eqref{eq:1.05}. In \eqref{admissible}, $h_{n_b}$ is defined via
 \begin{equation}\label{eq:1.01}
   h_{n_b}:=\frac{e^{\eta(m\cdot n_b)^2}}{\int_\BS e^{\eta(m\cdot n_b)^2} \ud m},
 \end{equation}
 for some $n_b$ satisfying
  \begin{equation}\label{eq:boundary}
    n_b\in H^1(\RR^d;\RR^3)~\text{with compact support and}~|n_b(x)|=1~\text{a.e. for}~x\in\O.
  \end{equation}
The parameter $\eta$ in \eqref{eq:1.01} is a constant depending on $\alpha$ and will be precised later. Loosely speaking, we shall choose $\eta$ such that \eqref{eq:1.01} is the only  global minimizer of the homogeneous Maier-Saupe energy \eqref{eq:1.02}.
The ``thin shell" $\O^\delta$ is defined as
 \begin{equation*}
   \O^\delta:=\{x\in\O~|~\operatorname{dist}(x,\p\O)\leq  \delta\}
 \end{equation*}
 where $\delta>0$ is  the parameter  characterizing the strength of boundary effect.
  For a liquid crystal material in a bounded domain $\O$, the study of  its static configuration predicted by Onsager's energy \eqref{eq:free energy-inhom-intro} should take into account the boundary effect. In our case, we   assume a strong anchoring boundary condition in \eqref{admissible} by prescribing the orientation of the molecules in a boundary layer that is slightly wider than $\sqrt{\ve}$,  the re-scaled length of the molecular.
More precisely,   we assume
 \begin{equation}\label{eq:1.20}
   \delta=\ve^{1/2-\sigma}
 \end{equation}
 for any fixed $\sigma\in(0,\frac 12)$. As the reader will see in the sequel, the non-local boundary condition in \eqref{admissible} will reduce to the `usual' strong anchoring boundary condition when $\ve\to 0$.

Throughout this work,  for any $f\in\mathscr{A}$, we shall denote
 \begin{equation}\label{extension1}
  \bar{f}(x,m)=\left\{\begin{array}{rl}
    h_{n_b}(x,m),&~x\in \RR^d\backslash\O ,\\
    f(x,m),&~x\in\O.
  \end{array}\right.
\end{equation}

Our first result is concerned with the critical point of \eqref{eq:free energy-inhom-intro}-\eqref{eq:On-interaction}:
\begin{Theorem}\label{thmcri}
Let $\alpha>7.5$ and $\eta$ be the largest root of equation
\begin{equation*}
  \alpha=\frac{\int_0^1{\mathrm{e}}^{\eta{z^2}}{\mathrm{d}}{z}}{\int_0^1z^2(1-z^2){\mathrm{e}}^{\eta{z^2}}{\mathrm{d}}{z}}.
\end{equation*}
Assume that $g$ satisfies (\ref{assump1}).
For each $\ve>0$, let $f^\ve\in \mathscr{A}$ be the critical point  corresponding to \eqref{eq:free energy-inhom-intro}-\eqref{eq:On-interaction} and $\delta$ is chosen as \eqref{eq:1.20}. If there exists an $\ve$-independent constant $C$ such that,
\begin{equation}\label{push2}
    \frac 1{\ve}\int_\Omega  \(A[f^\ve] -A[h_{n_b}]\) \ud x+\frac \alpha{2\ve }\int_{\RR^d\times\RR^d}\left|Q[\bar{f^\ve}](x)-Q[\bar{f^\ve}](y)\right|^2
g_\ve\(x-y\)\ud x\ud y\leq C,
  \end{equation}
  then, modulo   extraction of a subsequence,
   \begin{equation*}
    f^\ve\rightharpoonup f~\text{ weakly  in}~L^1(\O\times\BS),
  \end{equation*}
  where  $f(x,m)$ is given by
  \begin{equation*}
    f(x,m)=\frac{e^{\eta(m\cdot n(x))^2}}{\int_\BS e^{\eta(m\cdot n(x))^2} \ud m}
  \end{equation*}
  for some weakly harmonic map $n(x)\in H^1(\O;\BS)$ with boundary condition $n|_{\p\O}=\pm n_b|_{\p\O}$.
\end{Theorem}
\begin{Remark}
  Weakly harmonic map into $\BS$ is defined to be the weak solution to the Euler-Lagrange equation for \eqref{eq:harmonic}. More precisely, it satisfies
\begin{equation*}
 \sum_{j=1}^d\int_{\O}\p_j \varphi(x)\cdot(n(x)\wedge\p_j n(x))\ud x=0
\end{equation*}
for any $\varphi\in C_0^\infty(\O;\RR^3)$.
\end{Remark}
The conclusion in Theorem 1 can be strengthened if we assume  $f^\ve$ to be the solutions  of \eqref{globalmini2}:
\begin{Theorem}\label{thmglo}
Let $\alpha, \eta, g$  be the same as in Theorem 1.
For each $\ve>0$,  the minimizing problem \eqref{globalmini2} has a solution  $f^\ve \in \mathscr{A}$ which satisfies \eqref{push2} if additionally $\delta$  satisfies \eqref{eq:1.20}. Moreover, if $\delta$ is chosen as \eqref{eq:1.20} and $n_b|_{\O}\in H^1(\O;\BS)$ is a minimizing harmonic map\footnote{In this case, the boundary condition is prescribed by $n_b|_{\p\O}$.}, then  modulo  extraction of a subsequence,
\begin{equation*}
    f^\ve\rightharpoonup f~\text{ weakly  in}~L^1(\O\times\BS),
\end{equation*}
where  $f(x,m)$ is given by
\begin{equation*}
    f(x,m)=\frac{e^{\eta(m\cdot n(x))^2}}{\int_\BS e^{\eta(m\cdot n(x))^2} \ud m}
\end{equation*}
for some minimizing harmonic map $n(x)\in H^1(\O;\BS)$ with boundary condition $$n|_{\p\O}=\pm n_b|_{\p\O}.$$
\end{Theorem}

In contrast to the Ginzburg-Landau energy or the Landau-De Gennes energy, the Onsager's energy \eqref{eq:free energy-inhom-intro} is non-local by nature and its global minimizers do not possess nature energy estimate that is independent of $\ve$. However, by incorporating in it with
the strong anchoring boundary condition \eqref{admissible}, we prove that the global minimizers $f_\ve$  satisfies \eqref{push2}.
On the other hand, in order to recover the Oseen-Frank energy in the $\ve$-limit, we shall work with macroscopic parameter $Q[f_\ve]$ instead of the number density $f_\ve$ itself for the later  might not possess compactness.
The key step to obtain  the strong compactness of $Q[f_\ve]$ is to deduce from \eqref{push2} that, $g_\ve*Q[\bar{f}^\ve]$ has a uniform in $\ve$ bound in $H^1(\RR^d)$. As the reader shall see, during the process of passing to the $\ve$-limit, the $Q$-tensor will serve as an intermediate parameter connecting the number density function in molecule theory and the unit vector field in vector theory.

The rest of the article  is organized as follows. In
Section \ref{prelimi}, we  introduce the   notations and conventions that will be adopted throughout this work. In the meanwhile, some  assumptions will be made, for example, on  the kernel function $g$ in \eqref{eq:On-interaction} and the parameter $\eta$ in \eqref{eq:1.01}. Moreover, several analytic results concerning Maier-Saupe mean-field theory, especially the isotropic-nematic phase transition will be discussed. Section \ref{sec5} serves as a preparation for the proof of Theorem \ref{thmcri} where the Euler-Lagrange equation of \eqref{eq:free energy-inhom-intro} is derived and is recast  in terms of the macroscopic variable $Q[f]$. Moreover, in Proposition \ref{compactness1}, various weak and strong convergence results are obtained based on the a priori estimate \eqref{push2}. Section \ref{sec6} is devoted to the proof of Theorem \ref{thmcri}. The proof of Theorem \ref{thmglo} will be given in Section \ref{sec3} by first proving that \eqref{push2} is fulfilled by the solutions to \eqref{globalmini2}.

\section{Preliminary}\label{prelimi}
\subsection{Notation}
In this work, $\O\subset\RR^d$ will be a {\it{simply-connected}} domain which is decomposed as union $\O^\delta\cup \O_\delta$ where
   $$\O^\delta:=\{x\in\O~|~\operatorname{dist}(x,\p\O)\leq  \delta\},\quad \O_\delta:=\{x\in\O~|~\operatorname{dist}(x,\p\O)\geq  \delta\}.$$

 For any $k\times k$ symmetric matrix $M=\{M_{ij}\}_{1\leq i,j\leq k}$,  the $j$-th row vector will be denoted by   $M^j=\{M_{ij}\}_{1\leq i\leq k}$. For two such matrix $M$ and $N$, their inner product will be defined via, under Einstein summation convention,
 $M:N=M_{ij}N_{ij}$
and this induce the norm $|M|=\sqrt{M:M}$. When $i$ appears as superscript or subscript, it denotes an integer. On the other hand, we shall also use it to denote $\sqrt{-1}$ when it is multiplied by some quantities.

 We shall now introduce some formulas related to the rotational gradient operator, which will be employed to derive the Euler-Lagrange equation for \eqref{eq:free energy-inhom-intro}.
Let $ m =(m_1, m_2, m_3)^T\in \BS$ and $\nabla_ m $ be the gradient operator on the
unit sphere $\BS$. The rotational gradient operator $\mathcal{R}$ is defined
by
$$\mathcal{R}= m \wedge\nabla_{ m }.$$
Under spherical coordinate $(\theta, \phi)$, it can be written as
\begin{align*}
\CR=&(-\sin\phi\mathbf{i}+\cos\phi\mathbf{j})\partial_\theta
-(\cos\theta\cos\phi\mathbf{i}+\cos\theta\sin\phi\mathbf{j}-\sin\theta\mathbf{k})
\frac{1}{\sin\theta}\partial_\phi\\
\eqdefa&\ii{\CR}_1+\jj{\CR_2}+\kk{\CR_3}.
\end{align*}
The following properties can be easily verified, see for instance \cite{MR3366748}:
\begin{equation}\label{rotational}
 \begin{split}
   & \mathcal{R}\cdot\mathcal{R}=\Delta_\BS,\quad   \mathcal{R}_jm_k=-\epsilon^{jk\ell}m_\ell,\\
  & \mathcal{R}( m \cdot u)= m \wedge u,\quad
\mathcal{R}\cdot( m \wedge u)=-2 m \cdot u ,\\
 &\int_{\BS}\mathcal{R} f_1 f_2 \ud m =-\int_{\BS}f_1\mathcal{R} f_2 \ud m.
 \end{split}
\end{equation}
 Here $\Delta_\BS$ is the Laplace-Beltrami operator on $\BS$ and
$\varepsilon^{ jk\ell}$ is the Levi-Civita symbol, which is convenient in dealing with computation related to wedge product: for any $a,b\in\mathbb{R}^3$, one has
\begin{equation}\label{eq:1.12}
  a\wedge b=\{a_kb_\ell\varepsilon^{ jk\ell}\}_{1\leq j\leq 3}.
\end{equation}

\subsection{Convolution Operator and Interaction Kernel}
   We shall use the notation  $*_\O$ to denote the non-commutative convolution of $u,v$ in a domain $\O$:
\begin{equation*}
  u*_\O v: =\int_{\O}u(x')v(x-x')\ud x'.
\end{equation*}
Note that $u*_\O v=v*_\O u$ does not hold in general.
On the other hand,, we shall use   $*$ to denote the convolution of two functions  $u$ and $v$ in $\RR^d$:
\begin{equation*}
  u*v:=\int_{\RR^d}u(x-x')v(x')\ud x'=\int_{\RR^d}u(x')v(x-x')\ud x'.
\end{equation*}

For any function $v\in L^1(\RR^d)\cap L^2(\RR^d)$,  $\hat{v}$ will denote its  Fourier transform:
\begin{equation*}
  \hat{v}(\xi)=\int_{\RR^d}v(x)e^{-2\pi i x\cdot \xi}\ud x.
\end{equation*}
Then we have the following formulas (see \cite{SteinShakarchi2003})
\begin{equation*}
  \begin{split}
  &\hat{v}(x)=\int_{\RR^d}\hat{v}(\xi)e^{2\pi i x\cdot \xi}\ud \xi,\\
    &\widehat{\nabla u}(\xi)=2\pi i\xi\hat{u}(\xi),\\
    &\widehat{uv}(\xi)=\hat{u}(\xi)*\hat{v}(\xi),\\
    &\widehat{u*v}=\hat{u}(\xi)\hat{v}(\xi),\\
    &\|u\|_{L^2(\RR^d)}=\|\hat{u}\|_{L^2(\RR^d)}.
  \end{split}
\end{equation*}

Now we turn to the assumptions on the kernel function $g$ in \eqref{eq:On-interaction}.
We shall assume  $g(x)$ be a non-negative, smooth, {\bf radial function with exponential decay} and satisfies
$\int_{\RR^d}g(x) =1$. Moreover, we assume there exists some $c_0>0$ such that
\begin{equation}\label{assump1}
0\le\hat{g}(\xi)\le1~\text{and}~ c_0 |\xi|^2\hat{g}^2(\xi)\leq 1-\hat{g}(\xi) ,~\forall\xi\in\RR^d.
\end{equation}
 Under these assumptions, one can easily verify that
\begin{equation}\label{ftrans1}
  \hat{g}(0)=1 ,\quad \nabla\hat{g}(0)=0,\quad \nabla^2\hat{g}(0)=-\tfrac {4\pi^2\mu} d\mathbb{I}_d~\text{where}~\mu=\int_{\RR^d}|x|^2g(x)\ud x.
\end{equation}
Example of   $g$ satisfying the above assumptions is given by
$$g(x)=\(\tfrac a\pi\)^{\frac d2}e^{-a |x|^2}~\text{with}~a\in(0,\pi).$$
Actually, as $\hat{g}(\xi)=e^{-\tfrac{\pi^2 |\xi|^2}a}$, \eqref{assump1} holds with $c_0\leq \frac{\pi^2}a$.

We shall re-scale $g$ by $g_\ve=\frac 1{\sqrt{\ve}^d}g(\frac x{\sqrt{\ve}})$. Thus, we have
\begin{equation*}
\int_{\RR^d} g_\ve(x)\ud x=1,~\hat{g}_\ve(\xi)=\hat{g}(\sqrt{\ve}\xi).
\end{equation*}

  We give a few more words about the boundary condition introduced in \eqref{eq:boundary}. Any $ n_b\in H^1(\O;\BS)$  can be extend  to be
\begin{equation}\label{bound4}
\left\{
\begin{array}{rl}
 &n_b\in H^1(\RR^d;\RR^3),~\text{with}~n_b(x)\equiv 0~\text{for}~|x|\geq R~\text{where}~\O\subset B_{\frac R4}(0),\\
   &\|n_b\|_{W^1_p(\RR^d;\RR^3)}\leq C\|n_b\|_{H^1(\O)}~\text{with}~1\leq p\leq 2.
\end{array}
\right.
  \end{equation}
The existence of $n_b$  can be justified by first employing the extension theorem of Sobolev space (see for instance \cite[page 181]{MR0290095}) and then multipling by  a cutoff function
\begin{equation*}
  \chi(x)=\left\{
  \begin{array}{rl}
    1,&\quad |x|\leq  \frac R2 ,\\
    0,&\quad |x|\geq R.
  \end{array}
  \right.
\end{equation*}

\subsection{The Maier-Saupe energy}
Now we turn to the study of  the minimizers of the Maier-Saupe energy \eqref{Intro-energy}. It is proved
in \cite{fatkullin2005critical,MR2164198} that, all the critical points of \eqref{Intro-energy} can be explicitly given by
 \begin{equation}\label{uniaxial}
   h_{\nu}:=\frac{e^{\eta(m\cdot \nu)^2}}{\int_\BS e^{\eta(m\cdot \nu)^2} \ud m}
 \end{equation}
 for any $\nu\in\BS$
 where $\eta$ is a constant depending on the interaction strength $\alpha$:
\begin{align}\label{eta-alpha}
\frac{3{\mathrm{e}}^\eta}{\int_0^1{\mathrm{e}}^{\eta{z^2}}{\mathrm{d}}{z}}=3+2\eta+\frac{4\eta^2}{\alpha}.
\end{align}
The trivial solution   $\eta=0$ to \eqref{eta-alpha}  corresponds to the isotropic  distribution $h\equiv \frac{1}{4\pi}$.
If $\alpha>\alpha^*=\min_{\eta\in\mathbb{R}} \alpha(\eta)\approx6.7314$, then  \eqref{eta-alpha}  has nontrivial solutions satisfying
\begin{equation}\label{eq:1.10}
  \alpha=\alpha(\eta):=\frac{\int_0^1{\mathrm{e}}^{\eta{z^2}}{\mathrm{d}}{z}}{\int_0^1z^2(1-z^2){\mathrm{e}}^{\eta{z^2}}{\mathrm{d}}{z}}.
\end{equation}
In \cite{MR3366748}, it is proved that there exists a unique $\eta^*$ such that $\alpha(\eta^*)=\alpha^*$, and $\alpha(\eta)$ increases monotonically
when $\eta>\eta^*$ and decreases monotonically when $\eta<\eta^*$. Thus, for $\alpha>\alpha^*$, there exists two values, denoted by
$\eta_1(\alpha)$ and $\eta_2(\alpha)$, such that $\eta_1>\eta^*>\eta_2$. In addition, $\eta_1(\alpha)$ is an increasing function of $\alpha$, while $\eta_2(\alpha)$ is a decreasing function. It is also proved that, for $\alpha<7.5$, the critical point corresponding to $\eta=0$ is stable while for $\alpha>7.5$ it is unstable. For $\alpha>\alpha^*$, the critical points corresponding to $\eta_1$ are always stable and the ones corresponding to $\eta_2$ is unstable.

\begin{Lemma}\label{axisymmetric}
For $\alpha>7.5$, the global minimizer of \eqref{Intro-energy} in the function space $$\mathscr{H}:=\left\{f\in L^1(\BS)~|~f\geq 0~\text{a.e. on}~\BS,~\|f\|_{L^1(\BS)}=1\right\}$$
is achieved only by the uniaxial distribution \eqref{uniaxial} with  $\eta=\eta_1(\alpha)\neq 0$.
 \end{Lemma}
 The proof will make use of the following classical result concerning   the weakly lower-semicontinuity of   entropy:
\begin{Lemma}\label{convexlem}
  Let $f_k\in\mathscr{H}(\O)$ (defined by \eqref{eq:1.09}) be a sequence of functions such that
   \begin{equation*}
     \int_{\Omega\times\BS}f_k
\log f_k~ <\infty~\text{uniformly for}~k\in\mathbb{N}^*.
   \end{equation*}
    Then   there exists $f\in \mathscr{H}(\O)$ such that $f_k\rightharpoonup f$ weakly in $L^1(\BS\times\Omega)$ and
  \begin{equation}\label{convexlimit}
  \int_{\Omega\times\BS} f\log f\ud x\ud m
\leq\liminf_{k \to \infty}\int_{\Omega\times\BS}f_k\log f_k\ud x\ud m.
  \end{equation}
\end{Lemma}
\begin{proof}
  The weakly $L^1$-compactness of $\{f_k\}_{k\geq 1}$ and the almost everywhere inequality $f\geq 0$ follow from  \cite[page 47 and page 53]{MR2197021}. In order to show that $f\in\mathscr{H}(\O)$, we choose any test   function $\varphi=\varphi(x)$ and the weak convergence of $f_k$ leads to
\begin{equation*}
\int_\O \varphi(x)\ud x= \int_{\BS\times\O} f_k(x,m)\varphi(x)\ud x\ud m \to \int_{\BS\times\O}f(x,m)\varphi(x)\ud x\ud m
\end{equation*}
Since $\varphi(x)$ is arbitrary, we have that $$\int_\BS f(x,m)\ud m=1~a.e.~ x\in\O.$$
\end{proof}
\begin{proof}[Proof of Lemma \ref{axisymmetric}]
We first note that the space $\mathscr{H}$ in the statement of  Lemma \ref{axisymmetric} is non-empty since it includes the isotropic distribution $h=\frac 1{4\pi}$.
  Choosing any minimizing sequence $f_k\in\mathscr{H}$ such that $$\lim_{k\to\infty}A [f_k]= \inf_{f\in\mathscr{H}}  A[f] .$$
 Then it follows from   \eqref{eq:1.02} that
  $$ \int_\BS f_k\log f_k \ud m \leq C$$
  where $C$ is independent of $k$. Note that, we can consider $f_k$ to be elements in $\mathscr{H}(\O)$ that remain constant with respect to $x\in\O$. So we can apply Lemma \ref{convexlem} and this leads to
  \begin{equation*}
    f_k\rightharpoonup f~\text{ weakly  in}~L^1( \BS)
  \end{equation*}
 where  $f\in \mathscr{H}$.  This proves the existence of global minimizer $f$. Now we investigate its precise form.
  For any $\tilde{f}\in \mathscr{H}$ and $t\in [0,1]$, since $(1-t)f+t\tilde{f}\in \mathscr{H}$, the series $A[(1-t)f+t\tilde{f}]$ is well defined. Moreover, since $f$ is a global minimizer, we have
\begin{equation*}
 0\leq \lim_{  t\to 0^+}\frac {A[f+t(\tilde{f}-f)]-A[f]}  t  = \int_\BS (\log f+\CU[f] )(\tilde{f}-f) \ud m.
\end{equation*}
Since $\tilde{f}$ is arbitrary in $\mathscr{H}$, there exists some $\lambda\in\RR$ such that
\begin{equation*}
  \log f+\CU[f]= \lambda,\quad ~\text{a.e. on} ~ \BS.
\end{equation*}
One can easily determine $\lambda$ and deduce
\begin{equation*}
  f=\frac {e^{-\CU[f](m) }}{\int_\BS e^{-\CU[f](m) }\ud m}.
\end{equation*}
This implies that $f$ is also a critical point of \eqref{Intro-energy} and this together with  the discussion before leads to the desired result.
\end{proof}

In the sequel, we will work with $\alpha$ and $\eta$ as in Lemma \ref{axisymmetric}:
\begin{equation*}
  \alpha>7.5,\quad \eta=\eta_1(\alpha)\neq 0.
\end{equation*}
The relationship between the   uniaxial distribution \eqref{uniaxial} and its $Q$-tensor is nicely summarized by the  following  formulas:
\begin{Lemma}\label{lemma1}
For any uniaxial distribution \eqref{uniaxial} with $\nu\in \BS$, we have
  \begin{equation}\label{uniaxialtensor}
  Q[h_\nu]=s_2\(\nu\otimes \nu-\tfrac 13\mathbb{I}_3\)~\text{where}~s_2\neq 0.
\end{equation}
 For isotropic distribution $h=\frac 1{4\pi}$, it holds $$Q[h]=0.$$
\end{Lemma}
\begin{proof}
From \cite[Lemma 6.6]{MR3366748}, we have that
\begin{align*}
Q[h_\nu]=\int_{\BS}(m\otimes m-\tfrac13\mathbb{I}_3)h_\nu(m)\ud m=s_2(\nu\otimes\nu-\tfrac13\mathbb{I}_3).
\end{align*}
In the above formula, the parameter  $s_2$ is called degree of orientation and  is defined via
\begin{equation}\label{s2def}
  s_2=\int_\BS P_2(m\cdot \nu)\frac{e^{\eta(m\cdot \nu)^2}}{\int_\BS e^{\eta(m\cdot \nu)^2}\ud m}\ud m
  =\frac{\int_{-1}^1 P_2(z)e^{\eta z^2}\ud z}{\int_{-1}^1 e^{\eta z^2}\ud z},
\end{equation}
with $P_2(x)=\frac 12(3x^2-1)$ being the $2$-th Legendre polynomial.

 Concerning the sign of $s_2$, note that
\begin{align*}
\int_0^1z(1-z^2)\ud( e^{\eta z^2})+\int_0^1 e^{\eta z^2} \ud ( z(1-z^2))=e^{\eta z^2} z(1-z^2)|_0^1=0,
\end{align*}
we deduce that
\begin{equation*}
  s_2=\frac{\eta}{\alpha}\neq 0
\end{equation*}
by recalling \eqref{eq:1.10}, \eqref{s2def} and the choice of $\eta$. The case for $h=\frac 1{4\pi}$ is evident.
\end{proof}

\section{$\ve$-Independent A Priori Estimate and Compactness}\label{sec5}

\begin{Lemma}\label{yuning:eulerlag}
  For each $\ve>0$, let $f\in \mathscr{A}$ be a critical point  to \eqref{eq:free energy-inhom-intro}-\eqref{eq:On-interaction}, then $Q[f]$ satisfies the following
  \begin{equation}\label{eulerlag}
   \sum_{i=1}^3 Q^i[f ](x)\wedge (Q^i[f ]*_\O g_\ve)(x)=0,~\forall x\in\O_\delta
  \end{equation}
  where $Q^i[f]$ denotes the $i$-th row vector of    $Q[f]$.
\end{Lemma}

\begin{proof}
The proof consists of three parts:

\noindent $\bullet$ \textit{Deriving the Euler-Lagrange equation}:
First note that $\mathscr{A}$ is a convex set: if $f, g\in\mathscr{A}$, so does   $(1-t)f +tg\in\mathscr{A}$. Since $f $ is a critical point to \eqref{globalmini2}, it holds that
  \begin{equation*}
  0\leq \lim_{t\to 0^+}\tfrac 1t\(A_\epsilon[f+t(g-f)]-A_\epsilon[f]\)=\int_{\O\times\BS} (\log f+\CU_\epsilon[f])(g-f)\ud m\ud x.
\end{equation*}
Following the same argument as in the proof of Lemma \ref{axisymmetric} we can get
\begin{equation}\label{eq:1.21}
  \log f (x,m)+\CU_\epsilon[f ](x,m)=\lambda(x),~\forall x\in \O_\delta,
\end{equation}
where $\lambda(x)$ is the Lagrange multiplier corresponding to the constrain.
Thus we obtain the following formula
\begin{equation}\label{expl1}
  f =\frac {e^{-\CU_\epsilon[f ]}}{\int_\BS e^{-\CU_\epsilon[f ]}\ud m}\quad a.e. ~x\in \O_\delta.
\end{equation}
This implies that, the global minimizer of \eqref{globalmini2} is actually smooth function since $\CU_\ve$ is a regular convolution operator.

\noindent$\bullet$ \textit{Macroscopic equation:}
It follows from \eqref{rotational} that
\begin{equation*}
\CR\CU_\ve[f]= -2\alpha \int_\Omega\int_\BS m\cdot m'(m\wedge m') g_\ve(x-x')f(x',m')\ud x'\ud m'.
\end{equation*}
Then we have by \eqref{eq:1.21}  that
\begin{equation*}
  \begin{split}
    \CR f&=-f\CR\CU_\ve [f]\\
    &=2\alpha f \int_{\BS\times\O}m\wedge  m'( m\cdot m')g_\ve(x-x')f(m',x')\ud m'\ud x'\\
    &=2\alpha f \int_{\BS\times\O}m_i  m'_j \varepsilon^{ij\ell} m_k  m'_kg_\ve(x-x')f(m',x')\ud m'\ud x'\\
    &=2\alpha f \int_{\O}m_i m_k  \varepsilon^{ij\ell}   g_\ve(x-x')\(\tfrac 13\delta_{jk}+\int_\BS \(m'_j m'_k-\tfrac 13\delta_{jk}\)f(m',x')\ud m'\)\ud x'\\
    &=2\alpha f(x,m)  m_i   m_k\varepsilon^{ij\ell}    Q_{jk}[f]*_\O g_\ve.
  \end{split}
\end{equation*}
Integrate this identity over $\BS$ and use \eqref{rotational} lead to
\begin{equation*}
 \varepsilon^{ij\ell} Q_{ik}[f](x)( Q_{jk}[f]*_\O g_\ve)(x) =0,
\end{equation*}
which is equivalent to \eqref{eulerlag}, according to \eqref{eq:1.12}.
\end{proof}

\begin{Proposition}\label{compactness1}
Let  $f^\ve\in \mathscr{A}$ be    extended to $\bar{f}^\ve$ via \eqref{extension1} and satisfies
\begin{equation}\label{push1}
    \frac 1{\ve}\int_\Omega  \(A[f^\ve] -A[h_{n_b}]\) \ud x+\frac \alpha{2\ve }\int_{\RR^d\times\RR^d}\left|Q[\bar{f^\ve}](x)-Q[\bar{f^\ve}](y)\right|^2
g_\ve\(x-y\)\ud x\ud y\leq C,
  \end{equation}
  for some $\ve$ independent constant $C$. Then modulo the extraction of a subsequence,
\begin{equation}\label{eq:1.06}
  u_\ve(x):=Q[\bar{f^\ve}](x)
\end{equation}
has the following properties
 \begin{equation}\label{20150626claim1}
  \left\{
  \begin{array}{rl}
    u_\ve\to &\Psi~\text{strongly in }~L^2(\RR^d),\\
     \nabla (u_\ve*g_\ve)\rightharpoonup &\nabla \Psi~\text{weakly in }L_{loc}^2(\RR^d)
  \end{array}
  \right.
\end{equation}
where $\Psi$ satisfies
\begin{equation*}
  \Psi\in H^1(\RR^d)~\text{with compact support}.
\end{equation*}
Moreover, up to the extraction of a subsequence,
  \begin{equation}\label{weakcon3}
    \bar{f^\ve}\rightharpoonup \bar{f}~\text{ weakly  in}~L_{loc}^1(\RR^d\times\BS),
  \end{equation}
  where  $\bar{f}(x,m)$ is given by
  \begin{equation}\label{extension5}
    \bar{f}(x,m)=\left\{
    \begin{array}{rl}
      \frac {e^{\eta(m\cdot n(x))^2}}{\int_\BS e^{\eta(m\cdot n(x))^2}dm}&~\text{for}~ x\in \O,\\
      h_{n_b}(x,m)&~\text{for}~ x\in \RR^d\backslash\O
    \end{array}
    \right.
  \end{equation}
  for some $n(x)\in H^1(\O;\BS)$
   and
   \begin{equation}\label{eq:1.11}
     \Psi(x)=Q[\bar{f}](x)~a.e.~x\in \RR^d.
   \end{equation}
\end{Proposition}
\begin{proof}

The proof will be separated into several parts.

\noindent$\bullet$ \textit{Proof of  \eqref{20150626claim1}:}
 The assertion follows if we can prove the following estimate:
  \begin{equation}\label{20150626bound1}
\frac 1\ve\int_{\RR^d} |u_\ve*g_\ve-u_\ve|^2 \ud x+\int_{\RR^d} |\nabla (u_\ve*g_\ve)|^2 \ud x\leq  C.
\end{equation}
Actually, it follows from   \eqref{20150626bound1} and compact imbedding theorem of Sobolev space that, up to the extraction of a subsequence,
$\{u_\sigma*g_\sigma\}_{\sigma>0}$ is a Cauchy sequence in $L^2_{loc}(\RR^d)$ and this together with the   following inequality implies the strong convergence of $u_\ve$ in $L^2_{loc}(\O)$:
\begin{equation*}
  |u_\ve-u_\sigma|\leq |u_\ve-g_\ve*u_\ve|+|u_\sigma-u_\sigma*g_\sigma|+|u_\sigma *g_\sigma-u_\ve*g_\ve|.
\end{equation*}
Moreover, it follows from  \eqref{eq:boundary}, \eqref{eq:1.01} and \eqref{bound4} that,  $\bar{f}^\ve(x,m)\equiv\frac 1{4\pi}$ for $x\in\RR^d\backslash B_R(0)$. This together with Lemma \ref{lemma1} implies that
\begin{equation*}
  u_\ve(x):=Q[\bar{f^\ve}](x)\equiv 0,\quad \forall x\in\RR^d\backslash B_R(0).
\end{equation*}
So $u_\ve \to\Psi$ strongly in $L^2(\RR^d)$  where $\Psi\in L^2(\RR^d)$ with compact support.

The second part of \eqref{20150626claim1} follows from the weak compactness of $L^2_{loc}(\RR^d)$ and \eqref{20150626bound1}: on one hand, we have $\nabla (u_\ve* g_\ve)\rightharpoonup \Phi=\{\Phi_k\}_{1\leq k\leq d}$ weakly in $L^2_{loc}(\RR^d)$. Then for any test function $\varphi(x)\in C_0^\infty(\O)$,
\begin{equation*}
  -\int_{\RR^d}\p_k (u_\ve* g_\ve)\varphi=\int_{\RR^d} (u_\ve* g_\ve)\p_k\varphi=\int_{\RR^d} u_\ve \cdot(g_\ve*\p_k\varphi).
\end{equation*}
Taking $\ve\to 0$ leads to
\begin{equation*}
  -\int_{\RR^d}\varphi \Phi_k =\int_{\RR^d} \Psi\p_k\varphi~\text{which implies}~\nabla \Psi=\Phi\in L^2(\RR^d).
\end{equation*}

Now we are in position to prove \eqref{20150626bound1}. The reader can consult \cite{AB} for an approach without using Fourier transform. First, we have from Plancherel theorem that
\begin{equation}\label{bound1}
\begin{split}
\int_{\RR^d} |u_\ve*g_\ve-u_\ve|^2 \ud x
=&\|(1-\hat{g}_\ve)\hat{u}_\ve\|_{L^2}^2
\le 2\|\sqrt{1-\hat{g}_\ve}\hat{u}_\ve\|_{L^2}^2\\
=&\int_{\RR^d\times\RR^d} g_\ve(x-y)|u_\ve(x)-u_\ve(y)|^2\ud x\ud y.
\end{split}
\end{equation}
 On the other hand, it follows from \eqref{assump1} that we have
\begin{equation}\label{bound2}
  \begin{split}
  &\int_{\RR^d} |\nabla (g_\ve*u_\ve)|^2 \ud x
  =4\pi^2\|\xi\hat{g}_\ve\hat{u}_\ve\|_{L^2}^2\le C \ve^{-1}\|\sqrt{(1-\hat{g}_\ve)}\hat{u}_\ve\|_{L^2}^2\\
    &=\tfrac C\ve   \int_{\RR^d\times\RR^d}  |u_\ve(x)-u_\ve(y)|^2g_\ve(x-y) \ud x \ud y.
  \end{split}
\end{equation}
Then we can combine \eqref{bound1} with \eqref{bound2} together with \eqref{push1} to get \eqref{20150626bound1}.

\noindent$\bullet$ \textit{Proof of  \eqref{weakcon3}:}
It suffices to show
  \begin{equation}\label{extension6}
    \bar{f^\ve}\rightharpoonup \bar{f}~\text{ weakly  in}~L^1(\O\times\BS)
  \end{equation}
  for some  number density function $\bar{f}(x,m)$ which is uniaxial on $\O\times\BS$ and is extended to be  $h_{n_b}$ in $\O^c\times\BS$. Actually, it follows  from   \eqref{extension1}, \eqref{extension5}, and \eqref{extension6}   that, for any test function $\varphi(x,m)$,
  \begin{equation*}
    \begin{split}
      &\int_{\RR^d\times\BS}\bar{f^\ve}(x,m)\varphi(x,m)\ud x\ud m\\
      =&\int_{\O\times\BS}\bar{f^\ve}(x,m)\varphi(x,m)\ud x\ud m+\int_{ \O^c \times\BS}h_{n_b}(x,m)\varphi(x,m)\ud x\ud m\\
      \rightarrow &\int_{\O\times\BS}\bar{f}(x,m)\varphi(x,m)\ud x\ud m+\int_{ \O^c \times\BS}h_{n_b}(x,m)\varphi(x,m)\ud x\ud m\\
      =&\int_{\RR^d\times\BS}\bar{f}(x,m)\varphi(x,m)\ud x\ud m.
    \end{split}
  \end{equation*}

  To prove \eqref{extension6},
  we  first deduce from  \eqref{eq:1.02} and \eqref{push1} that
\begin{equation*}
  \int_{\O\times\BS}\bar{f^\ve}\ln \bar{f^\ve}\ud x\ud m+\frac23\alpha|\O|-\alpha\int_\O|Q[f^\ve](x)|^2\ud x=\int_\O A[f^\ve](x)\ud x\leq C
\end{equation*}
Owning to the fact that
\begin{equation*}
  |Q[f^\ve](x)|:=\left|\int_{\BS}\(m\otimes m-\tfrac 13\mathbb{I}_3\)f(x,m)\ud m\right|\leq 1,\quad \forall x\in \O,
\end{equation*}
 we obtain the entropy estimate
\begin{equation*}
  \int_{\O\times\BS}\bar{f^\ve}\ln \bar{f^\ve}\ud x\ud m\leq \tilde{C}.
\end{equation*}
This together with   Lemma \ref{convexlem} leads to \eqref{extension6} and thus
\eqref{weakcon3} is proved. It remains to show that $\bar{f}(x,m)$ is uniaxial on $\O$.

\noindent$\bullet$ \textit{Proof of  \eqref{extension5}:}
 To show that $\bar{f}(x,m)$ is uniaxial on $\O$, we deduce from \eqref{push1} that
\begin{equation*}
  \liminf_{\ve\to 0}\int_\Omega A[\bar{f^\ve}](x)\ud x\leq  \int_\O A[h_{n_b}](x) \ud x.
  \end{equation*}
  This together with strong compactness of $Q[\bar{f^\ve}](x)$ (see \eqref{20150626claim1}) and Lemma \ref{convexlem} lead to
  \begin{equation*}
  \int_\Omega A[\bar{f}](x)\ud x\leq   \int_\O A[h_{n_b}](x) \ud x
  \end{equation*}

Since $n_b$ is unit vector field on $\O$, $h_{n_b}$ is an uniaxial distributions on $\O$, which minimize the Maier-Saupe energy \eqref{Intro-energy} owning to Lemma \ref{axisymmetric}. So there exists some function $n(x):\O\mapsto \BS$ such that
\begin{equation*}
  \bar{f}(x,m)=\frac{e^{\eta(m\cdot n(x))^2}}{\int_\BS e^{\eta(m\cdot n(x))^2} \ud m}~\text{for}~x\in\O.
\end{equation*}
On the other hand, \eqref{weakcon3} also implies that
 \begin{equation*}
   Q[\bar{f^\ve}](x)\rightharpoonup Q[\bar{f}](x)~\text{ weakly in }~L_{loc}^1(\RR^d).
 \end{equation*}
 This together with \eqref{20150626claim1} implies that $$Q[\bar{f}](x)=\Psi(x)\in H^1(\O).$$ So $\bar{f}|_{\O\times\BS}$ is a uniaxial distribution    whose $Q$-tensor belongs to $H^1(\O)$. This together with the assumption that $\O$ is simply-connected enable us to apply   the orientability theorem in \cite[Theorem 2]{BallZarnescu2011} and  deduce that $n(x)\in H^1(\O;\BS).$
\end{proof}

 For any function $u(x)\in L^2(\RR^d)$, we define $\mathcal{A}_\ve u$ by
\begin{equation}\label{multiplier}
  \mathcal{A}_\ve u=\frac 1\ve (u-u* g_\ve).
\end{equation}
The operator $\mathcal{A}_\ve$ is a pseudo-differential operator with non-negative symbol
\begin{equation*}
\widehat{\mathcal{A}_\ve u}(\xi)=\frac{\hat{g}(0)-\hat{g}(\sqrt{\ve}\xi)}{\ve}  \hat{u},
\end{equation*}
as is seen from  \eqref{assump1} that $\hat{g}(0)-\hat{g}(\sqrt{\ve}\xi)\geq 0$. As a result we can define $h(\xi)$ as
\begin{equation}\label{lipschitz}
h(\xi):= \left\{\begin{array}{rl}
  \xi\sqrt{\frac{\hat{g}(0)-\hat{g}( \xi)}{  |\xi|^2}  },&\xi\neq 0,\\
  0,& \xi=0.
\end{array}\right.
\end{equation}

\begin{Lemma}\label{lipschitz1}
  The function $h(\xi)$ defined by \eqref{lipschitz} is globally Lipschitz in $\RR^d$.
\end{Lemma}
\begin{proof}
 It follows from \eqref{ftrans1} that  $h(\xi)$ is continuous at $\xi=0$ since $\lim_{\xi\to 0} h(\xi)=0$, according to \eqref{ftrans1}. On the other hand, $h(\xi)$ is smooth in $\RR^d\backslash\{0\}$ and decays to zero when $\xi\to \infty$. So $h\in L^\infty(\RR^d)\cap C(\RR^d)$.
 We compute the derivative of $h$ by
 \begin{equation*}
   \begin{split}
     \nabla h(\xi)=\mathbb{I}_d\sqrt{\frac{1-\hat{g}(\xi)}{|\xi|^2}}-\frac{\xi}{2\sqrt{ 1-\hat{g}(\xi) }} \otimes\(\frac{\nabla\hat{g}(\xi)}{|\xi|}+\frac {2\xi}{|\xi|^3}(1-\hat{g}(\xi))\)=\sum_{k=1}^3 A_i(\xi),\forall \xi\neq 0.
   \end{split}
 \end{equation*}
 It is evident that $A_1,A_3\in L^\infty(\RR^d)\cap C(\RR^d)$. Moreover, $A_2\in L^\infty(B_1)\cap C^\infty(\RR^d\backslash B_1)$ and decays to zero as $\xi\to \infty$. These all together imply the statement.
\end{proof}
Therefore, we shall define a  vector valued operator $\mathcal{T}_\ve=\{T^i_\ve\}_{1\leq i\leq d}$ by
\begin{equation}\label{defteps}
\widehat{\mathcal{T}_\ve u}
=\xi\sqrt{\frac{\hat{g}(0)-\hat{g}(\sqrt{\ve}\xi)}{\ve |\xi|^2}  } \hat{u}=\frac1{\sqrt\ve} h(\sqrt\ve \xi)\hat{u},
\end{equation}
and  we have
\begin{equation}\label{squrtroot}
   \mathcal{A}_\ve  =\sum_{  k=1}^d\mathcal{T}^k_\ve \circ \mathcal{T}^k_\ve .
 \end{equation}
The symbol of $\mathcal{T}_\ve$ will approach $\xi$ as $\ve\to 0$.
The following lemma says that the pseudo-differential operator $\mathcal{T}_\ve$ will approach, as $\ve\to 0$, to $\frac 1i\sqrt{\tfrac \mu d}\nabla$.
\begin{Lemma}\label{converge-T}
If $u\in H^1(\RR^d)$, then it holds
\begin{equation*}
  \mathcal{T}_\ve u \rightarrow -i\sqrt{\tfrac \mu d} \nabla u~\text{ strongly in }L^2(\RR^d).
\end{equation*}
\end{Lemma}
\begin{proof}
It can be verified from \eqref{ftrans1} that
\begin{equation*}
  \sqrt{\frac{1-\hat{g}( \sqrt{\ve}\xi)}{ \ve |\xi|^2}  }~\text{is uniformly bounded with respect to }~\ve>0~\text{and}~\xi\in \RR^d\backslash\{0\}
\end{equation*}
and
\begin{equation*}
  \lim_{\ve\to 0}\sqrt{\frac{1-\hat{g}( \sqrt{\ve}\xi)}{ \ve |\xi|^2}  }=2\pi\sqrt{\frac \mu d}\qquad\quad~\text{for point-wise}~\xi\in\RR^d\backslash\{0\}.
\end{equation*}
On the other hand, since $u\in H^1(\RR^d)$, we have $\|\xi \hat{u}(\xi)\|_{L^2(\RR^d)}<\infty.$
Therefore, Lebesgue's dominant convergence theorem implies
\begin{equation*}
   \lim_{\ve\to 0}\left\| (\mathcal{T}_\ve+\sqrt{\tfrac\mu d} i\nabla)u(x)\right\|_{L^2(\RR^d)}
   = \lim_{\ve\to 0}
   \left\|  \(\sqrt{\tfrac{1-\hat{g}( \sqrt{\ve}\xi)}{ \ve |\xi|^2}  }-2\pi\sqrt{\tfrac \mu d}\)\xi\hat{u}(\xi) \right\|_{L^2(\RR^d)}=0.
\end{equation*}
 The proof is finished.
\end{proof}

\begin{Lemma}\label{variation5}
Under the  assumptions of Proposition \ref{compactness1},  up to the extraction
of a subsequence,
\begin{equation*}
  \mathcal{T}_\ve Q[\bar{f^\ve}]\rightharpoonup -i\sqrt{\tfrac \mu d} \nabla Q[\bar{f}]~\text{ weakly in }L^2_{loc}(\RR^d),
\end{equation*}
where $\bar{f}$ is defined by \eqref{weakcon3}.
\end{Lemma}
\begin{proof}
First note that,  the uniform bound \eqref{push1} and the definition of $\mathcal{T}_\ve$ at \eqref{defteps} imply
\begin{equation*}
    \alpha\|\mathcal{T}_\ve Q[\bar{f^\ve}]\|^2_{L^2(\RR^d)} \leq C.
\end{equation*}
Then from weakly compactness of $L^2$-space, there exists $\widetilde{Q}$ such that
\begin{equation*}
  \mathcal{T}_\ve Q[\bar{f^\ve}]\rightharpoonup \widetilde{Q}(x)~\text{ weakly in }L^2_{loc}(\RR^d),
\end{equation*}
or equivalently,
\begin{equation*}
  \int_{\RR^d}\mathcal{T}_\ve Q[\bar{f^\ve}](x):\Phi(x)\ud x\rightarrow\int_{\RR^d}\widetilde{Q}(x):\Phi(x)\ud x,\quad \forall\Phi\in C_0^\infty(\RR^d;\RR^{d\times d\times d}).
\end{equation*}
On the other hand, the strong convergence of $Q[\bar{f^\ve}](x)$ stated in \eqref{20150626claim1} and Lemma \ref{converge-T} imply
\begin{equation*}
  \begin{split}
    &\int_{\RR^d}\mathcal{T}_\ve Q[\bar{f^\ve}](x):\Phi(x)\ud x\\
    =&\int_{\RR^d} Q[\bar{f^\ve}](x):\mathcal{T}_\ve\Phi(x)\ud x\rightarrow-i\sqrt{\tfrac\mu d}\int_{\RR^d}Q[\bar{f}](x):(\nabla\cdot \Phi(x))\ud x.
  \end{split}
\end{equation*}
The above two formulaes together imply  $\widetilde{Q}(x)=-i\sqrt{\tfrac \mu d}\nabla Q[\bar{f}](x)$. The proof is finished.
\end{proof}

\section{Asymptotic behavior of critical points}\label{sec6}

This section is devoted to the proof of   Theorem \ref{thmcri}.
We start from a commutator estimate:
\begin{Lemma}\label{lem:commu-estimate}
For any $\varphi\in C_0^\infty(\O)$, there exists a constant $C$ depending on $\varphi(x)$ such that
\begin{equation}\label{convolu1}
  \|[\mathcal{T}_\ve, \varphi(x)] u \|_{L^2(\RR^d)}\leq C\|u\|_{L^2(\RR^d)}.
\end{equation}
\end{Lemma}
\begin{proof}
By the definition of the commutator, we have
 $$[\mathcal{T}_\ve,\varphi(x)] u=\mathcal{T}_\ve(\varphi(x)u(x))-\varphi(x)\mathcal{T}_\ve u(x).$$
Then it follows from  Plancherel's theorem, Lemma \ref{lipschitz1} and Young's inequality that
\begin{equation*}
    \begin{split}
        &\|[\mathcal{T}_\ve,\varphi(x)] u\|_{L^2(\RR^d)}\\
        =&\frac 1{\sqrt{\ve}}\|  h(\sqrt\ve \xi)\hat{\varphi}*\hat{u}-\hat{\varphi}*( h(\sqrt\ve \xi) \hat{u}(\xi) )\|_{L^2(\RR^d)}\\
        =&\frac 1{\sqrt{\ve}}\left\|   h(\sqrt\ve \xi)\int_{\RR^d}\hat{\varphi}(\xi-\zeta) \hat{u}(\zeta)\ud \zeta-\int_{\RR^d}\hat{\varphi}(\xi-\zeta)  h(\sqrt\ve \zeta) \hat{u}(\zeta) \ud \zeta\right\|_{L^2(\RR^d)}\\
        =&\frac 1{\sqrt{\ve}}\left\| \int_{\RR^d}\hat{\varphi}(\xi-\zeta)(  h(\sqrt\ve \xi)-  h(\sqrt\ve \zeta)) \hat{u}(\zeta) \ud \zeta\right\|_{L^2(\RR^d)}\\
        \leq &\frac C{\sqrt{\ve}}\left\| \int_{\RR^d}\hat{\varphi}(\xi-\zeta)\sqrt{\ve} |\xi-\zeta|  \hat{u}(\zeta) \ud \zeta\right\|_{L^2(\RR^d)}\\
        = &C\left\|  (|\xi|\hat{\varphi}(\xi))* \hat{u} \right\|_{L^2(\RR^d)}\\
        \leq & C\||\xi|\hat{\varphi}(\xi)\|_{L^1(\RR^d)}\|\hat{u}\|_{L^2(\RR^d)}.
    \end{split}
  \end{equation*}
\end{proof}

\begin{Lemma}\label{lem:commu-convergence}
Under the assumption of Proposition \ref{compactness1},   for any $\varphi\in C_0^\infty(\O)$
\begin{equation}\label{claim1}
  [\mathcal{T}_\ve, \varphi(x)] Q[\bar{f^\ve}]\mapsto -i\sqrt{\tfrac\mu d}[ \nabla,\varphi(x)]Q[\bar{f}]~\text{strongly in}~L^2(\RR^d).
\end{equation}
\end{Lemma}
\begin{proof}
We have
\begin{equation*}
  \begin{split}
 & ~[\mathcal{T}_\ve, \varphi(x)] Q[\bar{f^\ve}]+[i\sqrt{\tfrac \mu d} \nabla,\varphi(x)]Q [\bar{f}]\\
  =&~  {[\mathcal{T}_\ve, \varphi(x)] \(Q [\bar{f^\ve}]-  Q [\bar{f}]\)}+ {[\mathcal{T}_\ve, \varphi(x)] Q [\bar{f}]+ [i\sqrt{\tfrac \mu d} \nabla, \varphi(x)] Q [\bar{f}]}\\
=&~ {[\mathcal{T}_\ve, \varphi(x)] \(Q [\bar{f^\ve}]-  Q [\bar{f}]\)}+(\mathcal{T}_\ve+i\sqrt{\tfrac \mu d} \nabla) \Big(\varphi(x)  Q[\bar{f}](x)\Big)-\varphi(x)\Big(\mathcal{T}_\ve+i\sqrt{\tfrac \mu d} \nabla\Big)Q[\bar{f}](x)\\
  \triangleq &~ I_1+I_2+I_3.
  \end{split}
\end{equation*}
The estimate of $I_1$ follows from   the commutator estimate \eqref{convolu1} and Proposition \ref{compactness1}: there exists constant $C$ depending on $\varphi(x)$ such that
\begin{equation*}
 \begin{split}
    \|[\mathcal{T}_\ve, \varphi(x)] \(Q [\bar{f^\ve}]-  Q [\bar{f}]\)\|_{L^2(\RR^d)}
    \leq  C\|Q [\bar{f^\ve}]-  Q [\bar{f}]\|_{L^2(\RR^d)} \to 0~\text{as}~\ve\to 0.
 \end{split}
\end{equation*}
To treat $I_2$, it follows \eqref{20150626claim1} that $Q[\bar{f}](x)=u(x)\in H^1(\RR^d)$ and then $\varphi Q[\bar{f}]\in H^1(\RR^d)$. Thus we
deduce from Lemma \ref{converge-T} that
$$\lim_{\ve\to0}\|I_2\|_{L^2(\RR^d)}=0.$$
The term $I_3$ can be estimated in the same way and proof is completed.
\end{proof}

Now we are ready to prove   Theorem \ref{thmcri}. Note that we choose $\delta$ as \eqref{eq:1.20}.
\begin{proof}[Proof of Theorem \ref{thmcri}]
Note that, owning to Proposition \ref{compactness1},
\begin{equation*}
    \bar{f^\ve}\rightharpoonup \bar{f}~\text{ weakly  in}~L^1_{loc}(\RR^d\times\BS),
  \end{equation*}
  where  $\bar{f}(x,m)$ is given by
  \begin{equation*}
    \bar{f}(x,m)=\left\{
    \begin{array}{rl}
      \frac {e^{\eta(m\cdot n(x))^2}}{\int_\BS e^{\eta(m\cdot n(x))^2}dm}&~\text{for}~ x\in \O,\\
      h_{n_b}(x,m)&~\text{for}~ x\in \RR^d\backslash\O
    \end{array}
    \right.
  \end{equation*}
  for some $n(x)\in H^1(\O;\BS)$.

\noindent$\bullet$ \textit{Proof of that $n(x)$ is a weakly harmonic map:}

We shall work with test function $\varphi(x)\in C_0^\infty(\O;\RR^d)$.
It is easy to see that, there exists some $\epsilon_0>0$ such that
$\varphi(x)\in C_0^\infty(\O_{\delta(\ve)})$ for all $\ve<\ve_0$. We shall assume in the sequel that $\ve<\ve_0$.
We will prove that, $Q[\bar{f}]$, with $\bar{f}$ given  in \eqref{weakcon3}, is a weak solution to the following equation:
\begin{equation*}
  \div ( Q^i[\bar{f}](x)\wedge \nabla  Q^i[\bar{f}](x))=0,\qquad \forall x\in \O.
\end{equation*}
Here and in the sequel, we adopt the Einstein's convention on the summation over repeat subscript.
Since $f^\ve$ are assumed to be critical points, it follows from Lemma \ref{yuning:eulerlag} that
\begin{equation*}
  Q^i[f^\ve](x)\wedge \frac1\ve \big(Q^i[f^\ve](x)-(Q^i[f^\ve]*_\O g_\ve)(x)\big)=0,~\forall x\in\O_{\delta(\ve)},
\end{equation*}
or equivalently, due to the convention \eqref{extension1},
\begin{equation*}
  \begin{split}
    &Q^i[\bar{f^\ve}](x)\wedge \frac1\ve \(Q^i[\bar{f^\ve}](x)-( Q^i[\bar{f^\ve}]*g_\ve)(x)\)\\
    =&-Q^i[f^\ve](x)\wedge\frac 1\ve \int_{\RR^d\backslash \O}g_\ve(x-x')Q^i[\bar{f^\ve}](x')\ud x',~\forall x\in \O_{\delta(\ve)}.
  \end{split}
\end{equation*}

According to \eqref{extension1}, we get
\begin{equation}\label{eq:1.13}
  \begin{split}
    &Q^i[\bar{f^\ve}](x)\wedge \mathcal{A}_\ve Q^i[\bar{f^\ve}])(x) \\
    =&-\frac 1\ve Q^i[f^\ve](x)\wedge\int_{\RR^d\backslash \O}g_\ve(x-x')Q^i[h_{n_b}](x')\ud x',~\forall x\in \O_{\delta(\ve)}.
  \end{split}
\end{equation}
Denote
\begin{equation}\label{eq:1.14}
  D_\ve:=-\frac 1\ve\int_{\RR^d}\varphi(x)\cdot\(Q^i[f^\ve](x)\wedge \int_{\RR^d\backslash \O}g_\ve(x-x')Q^i[h_{n_b}](x')\ud x'\)\ud x.
\end{equation}
From the exponential decay of $g_\ve$, we have
  \begin{equation}\label{decay-g}
    g_\ve(x-x')\leq  \frac C{\sqrt{\ve}^3}e^{-\frac \delta{\sqrt{\ve}}}\leq C\ve^{\frac 32},~\text{for}~ |x-x'|\geq \delta(\ve).
  \end{equation}
As a result, for any $x\in\Omega_\delta(\ve)$ and any $x'\in\RR^d\backslash \O $, we  have $|x-x'|\geq \delta(\ve)$ and thus \eqref{decay-g} implies
\begin{equation*}
  \begin{split}
  |D_\ve|=&\left|\frac 1\ve\int_{\Omega_\delta(\ve)}\varphi(x)\cdot\(Q^i[f^\ve](x)\wedge \int_{\RR^d\backslash \O}g_\ve(x-x')Q^i[h_{n_b}](x')\ud x'\)\ud x\right|\\
   \leq &C\ve^{\frac 12} \left|\int_{\RR^d\backslash \O}Q^i[ h_{n_b} ](x)dx\right|\left|\int_{\O_\delta}\varphi(x) Q^i[f^\ve](x)\ud x\right|.
      \end{split}
\end{equation*}
On the other hand, it follows from \eqref{bound4} and the second part of Lemma \ref{lemma1} that $Q^i[ h_{n_b} ]$ has compact support. Thus
\begin{equation*}
  \lim_{\ve\to 0}|D_\ve|= 0
\end{equation*}
and this together with  \eqref{eq:1.13}, \eqref{eq:1.14} lead to
\begin{align}\label{eq:lim-right}
\lim_{\ve\to0}\int_{\RR^d} \varphi(x)\cdot \(Q^i[\bar{f^\ve}](x)\wedge \mathcal{A}_\ve Q^i[\bar{f^\ve}](x)\)\ud x=0.
\end{align}

In order to obtain harmonic map, we need to manipulate the integrand inside \eqref{eq:lim-right}:
\begin{equation*}
  \begin{split}
  &\int_{\RR^d} \varphi(x)\cdot \(Q^i[\bar{f^\ve}](x)\wedge \mathcal{A}_\ve Q^i[\bar{f^\ve}](x)\)\ud x \\
    =&\int_{\RR^d} \varphi(x)\cdot \( Q^i[\bar{f^\ve}](x)\wedge \mathcal{T}^k_\ve\circ\mathcal{T}^k_\ve Q^i[\bar{f^\ve}](x)\) \ud x\\
    =&\int_{\RR^d} \mathcal{T}^k_\ve\( \varphi(x) Q^i[\bar{f^\ve}](x)\)\wedge \mathcal{T}^k_\ve Q^i[\bar{f^\ve}](x)\ud x\\
    =&\int_{\RR^d} [\mathcal{T}^k_\ve, \varphi(x)]\cdot\( Q^i[\bar{f^\ve}](x)\wedge \mathcal{T}^k_\ve Q^i[\bar{f^\ve}](x)\)\ud x+ \int_{\RR^d} \varphi(x)\cdot\(\mathcal{T}^k_\ve Q^i[\bar{f^\ve}](x)\wedge \mathcal{T}^k_\ve Q^i[\bar{f^\ve}](x)\)\ud x\\
    =&\int_{\RR^d} [\mathcal{T}^k_\ve, \varphi(x)]\cdot\( Q^i[\bar{f^\ve}](x)\wedge \mathcal{T}^k_\ve Q^i[\bar{f^\ve}](x)\) \ud x.
  \end{split}
\end{equation*}
The above formula can also be verified tediously using   components of various tensors as well as formula \eqref{eq:1.12}.
Taking $\ve\to 0$ in the above formula and employing Lemma \ref{variation5} and Lemma \ref{lem:commu-convergence}, we obtain
\begin{align*}
\lim_{\ve\to 0}\int_{\RR^d} \varphi(x)\cdot Q^i[\bar{f^\ve}](x)\wedge \mathcal{A}_\ve Q^i[\bar{f^\ve}](x)\ud x=
  -\frac{\mu}{d}\int_{\RR^d} [\p_j, \varphi(x)]\cdot\( Q^i[\bar{f}]\wedge \p_j Q^i[\bar{f}]\)\ud x.
\end{align*}
The together with  \eqref{eq:lim-right} lead to
\begin{equation*}
 \int_{\RR^d}\p_j \varphi(x)\cdot( Q^i[\bar{f}]\wedge\p_j  Q^i[\bar{f}])=0,
\end{equation*}
or equivalently, in terms of the components of matrix $Q=\{Q_{ij}\}_{1\leq i,j\leq d}$:
\begin{equation}\label{qtensorequ}
 \int_{\O}\p_j \varphi_\ell(x)\varepsilon^{\ell kq}  Q_{ik}[\bar{f}] \p_j  Q_{iq}[\bar{f}]=0.
\end{equation}
Since $\bar{f}$ is uniaxial in $\O$ according to \eqref{extension5}, it follows from \eqref{uniaxialtensor} that $$Q_{ik}(x)=s_2\(n_i(x)n_k(x)-\tfrac 13\delta_{ik}\).$$ Plugging this formula into \eqref{qtensorequ} leads to
\begin{equation*}
 s^2_2\int_{\O}\p_j \varphi(x)\cdot(n(x)\wedge\p_j n(x))=0
\end{equation*}
which is the weak formulation of the  harmonic map equation since $s_2\neq 0$ (see Lemma \ref{lemma1}).
\medskip

\noindent$\bullet$ \textit{Boundary Condition}:
It follows from \eqref{20150626claim1}, \eqref{eq:1.11} that
\begin{equation*}
 \nabla( Q[\bar{f^\ve}]*g_\ve) \rightharpoonup \nabla Q[\bar{f}] ~\text{weakly in}~ L^2_{loc}(\O)
\end{equation*}
 and   continuous embedding $H^1_{loc}(\RR^d)\hookrightarrow H^{\frac12}(\p\O)$  implies that
\begin{equation}\label{eq:1.15}
  Q[\bar{f^\ve}]*g_\ve(x)\rightharpoonup Q[\bar{f}](x) ~\text{weakly in}~ H^{\frac 12}(\p\O).
\end{equation}
On the other hand, it follows from \eqref{extension1} and \eqref{admissible} that
\begin{equation}\label{long1}
  \begin{split}
    &Q[\bar{f^\ve}]*g_\ve(x)\\
    =&\int_{\RR^d}Q[\bar{f^\ve}](x')g_\ve(x-x')\ud x'\\
    =&\int_{\O^c\cap \O^\delta\cup \O_\delta}Q[\bar{f^\ve}](x')g_\ve(x-x')\ud x'\\
    =&\int_{\O^c\cup \O^\delta}Q[h_{n_b}](x')g_\ve(x-x')\ud x'+\int_{\O_\delta}Q[\bar{f^\ve}](x')g_\ve(x-x')\ud x'\\
    =&\int_{\RR^d}Q[h_{n_b}](x')g_\ve(x-x')\ud x'+\int_{\O_\delta}(Q[\bar{f^\ve}](x')-Q[h_{n_b}](x'))g_\ve(x-x')\ud x'.
  \end{split}
\end{equation}
Given $x\in\p\O$, we have $|x-x'|\geq \delta(\ve)$ for any $x'\in \O_{\delta(\ve)}$. Thus, from the decay estimate  \eqref{decay-g}  for $g$, we get
 \begin{equation*}
  \lim_{\ve\to 0} \int_{\O_{\delta(\ve)}}(Q[\bar{f^\ve}](x')-Q[h_{n_b}](x'))g_\ve(x-x')\ud x'=0.
 \end{equation*}
 So we obtain from \eqref{long1} that
\begin{equation}\label{eq:1.16}
  \lim_{\ve\to 0}Q[\bar{f^\ve}]*g_\ve(x)=Q[h_{n_b}]\quad ~a.e.~ x\in\p\O.
\end{equation}
It follows from \eqref{eq:1.15} and \eqref{eq:1.16} that
 \begin{equation*}
   Q[\bar{f}](x)=Q[h_{n_b}](x)=s_2\(n_b(x)\otimes n_b(x)-\tfrac 13 \mathbb{I}_3\),\quad \forall x\in\p\O.
 \end{equation*}
 This together with the fact that
 $$Q[\bar{f}](x)=s_2\(n(x)\otimes n(x)-\tfrac 13 \mathbb{I}_3\),\quad\forall x\in \O$$
implies the boundary condition
 \begin{equation*}
   n(x)=\pm n_b(x),\quad a.e.~  x\in\p\O.
 \end{equation*}

\end{proof}
\section{Asymptotic behavior of global minimizers}\label{sec3}
The task of this section is to prove  Theorem \ref{thmglo}.
We shall first prove the existence of solutions to the minimizing problem:
\begin{Theorem}\label{existence}
  For each $\ve>0$, the   minimizing problem   \eqref{globalmini2} has a solution $f^\ve $ in the admissible class $ \mathscr{A}$.
\end{Theorem}
\begin{proof}
The proof will be divided into several parts:

\medskip

\noindent$\bullet$\textit{Minimizing sequence:}
 We first note that the space $\mathscr{A}$ defined by \eqref{admissible} is non-empty since the function $h_{n_b}(x,m)$ defined by \eqref{uniaxial} is in $\mathscr{A}$.
  Choosing any minimizing sequence $f_k\in\mathscr{A}$ such that $$\lim_{k\to\infty}  A_\ve[f_k] = \inf_{f\in\mathscr{A}}  A_\ve[f].$$
 Then it follows from \eqref{eq:free energy-inhom-intro} that
  $$\int_\O\int_\BS f_k\log f_k \ud m\ud x\leq C$$
  and this together with Lemma \ref{convexlem} leads to, up to the extraction of a subsequence,
  \begin{equation}\label{weakcon2}
    f_k\rightharpoonup f~\text{ weakly  in}~L^1(\Omega\times\BS)
  \end{equation}
 where  $f\in \mathscr{H}(\O)$.
To prove that $f\in\mathscr{A}$, it follows from \eqref{weakcon2} and the  continuity of the operator $Q$ that,
   \begin{equation}\label{weakcon4}
     Q[f_k]\rightharpoonup Q[f]~\text{ weakly in }L^1(\O).
   \end{equation}
 In order to     verify the condition in \eqref{admissible}, we note that, for fixed $\ve>0$, $g_\ve(x-x')$ is a smooth kernel. Combining this fact with $f_k\in\mathscr{A}$ as well as \eqref{weakcon4}, we arrive at
\begin{equation*}
\begin{split}
  \int_{\O^\delta} s_2 \(n_0(x)\otimes n_0(x)-\tfrac 13\mathbb{I}_3\):\varphi(x)\ud x&=\lim_{k\to\infty} \int_{\O^\delta} Q[f_k](x): \varphi(x)\ud x\\
  &= \int_{\O^\delta} Q[f](x): \varphi(x)\ud x,\quad \forall \varphi\in C_0^\infty(\O^\delta;\RR^{d\times d})
\end{split}
\end{equation*}
 and thus $f\in\mathscr{A}$.

\medskip

\noindent$\bullet$\textit{Proof that $f$ is a global minimizer:} Note that, for fixed $\ve$, the operator \eqref{inhomopot} is a regular integral operator. So up to the extraction of a subsequence,
  \begin{equation*}
    \lim_{k\to \infty}\CU_\ve [f_k](x,m)= \CU_\ve [f](x,m)~\text{strongly in}~C(\overline{\O\times\BS}).
  \end{equation*}
  This together with \eqref{convexlimit} implies the weakly lower-semi-continuity of $A_\ve$:
\begin{equation*}
   \begin{split}
       A_\ve[f]&=\int_\O\int_{\BS}\(f \log f
+f\CU_\ve [f]\)\ud x\ud m\\
     &\leq \liminf_{k \to \infty}\int_{\Omega\times\BS}f_k
\log{f_k} \ud x\ud m+\lim_{k\to \infty}\int_\O\int_{\BS}f_k\CU_\ve [f_k]\ud x\ud m\leq \liminf_{k\to\infty}A_\ve[f_k].
   \end{split}
\end{equation*}
\end{proof}

The following lemma implies that the interaction energy in \eqref{eq:free energy-inhom-intro} can be expressed in terms of the $Q$-tensor of the number density function:
\begin{Lemma}
The convolution type potential $\eqref{eq:free energy-inhom-intro}$ can be written by
  \begin{equation}\label{variation3}
 \begin{split}
A_\ve[f] =&\int_{\O\times\BS}f\log f \ud x\ud m-\alpha\int_\O|Q[f](x)|^2\ud x \\
  &+\alpha\int_{\O}M[f](x):\Big(Q[f](x)- (Q[f]*_\O g_\ve)(x)\Big)\ud x+C_1(g_\ve,\O),
  \end{split}
\end{equation}
where $C_1(g_\ve,\O)$ are explicit constant that are independent of $f$.
\end{Lemma}
\begin{proof}
It suffices to consider the   interaction part:
\begin{equation*}
  \begin{split}
    &  \alpha^{-1}\int_{\O\times\BS}f(x',m')\CU_\ve[f](x,m)\ud m\ud x \\
    = & \int_{\Omega\times\O}\int_{\BS\times\BS}f(x',m')f(m,x)|m\times m'|^2g_\ve\(x- x'\)\ud m\ud m'\ud x\ud x'\\
    = & \int_{\Omega\times\O}\int_{\BS\times\BS}f( x', m')f( x,m)\(1-| m\cdot m'|^2\)g_\ve\( x- x'\)\ud m\ud m'\ud x\ud x'\\
    = & \int_{\Omega\times\O}\int_{\BS\times\BS}f( x', m')f( x,m)\(\frac23-(m\otimes m-\frac13\mathbb{I}_3)
    :(m'\otimes m'-\frac13\mathbb{I}_3)\)g_\ve\( x- x'\)\ud m\ud m'\ud x\ud x'\\
    =&\frac{2}{3}\int_{\O\times \O} g_\ve(x-x')\ud x\ud x'-\int_\O|Q[f](x)|^2\ud x\\
    &\qquad+\int_{\Omega\times\O}Q[f](x):(Q[f](x)-Q[f](x'))g_\ve\(x-x'\)\ud x'\ud x.
      \end{split}
\end{equation*}
This together with \eqref{eq:free energy-inhom-intro} implies
\begin{equation*}
 \begin{split}
    A_\ve[f] =&\int_{\O\times\BS}f\log f \ud x\ud m-\alpha\int_\O|Q[f](x)|^2\ud x \\
  &+\alpha\int_{\O}Q[f](x):(Q[f](x)- (Q[f]*_\O g_\ve)(x))\ud x+C_1(g_\ve,\O)
  \end{split}
\end{equation*}
with
\begin{equation*}
  C_1(g_\ve,\O)=\frac{2\alpha}{3}\int_{\O\times \O} g_\ve(x-x')\ud x\ud x'.
\end{equation*}
\end{proof}

  \begin{Lemma}\label{eq:1.03}
  Let $\alpha, \eta, g$  be the same as in Theorem 1 and $\delta$ satisfy \eqref{eq:1.20} for some $\sigma\in(0,\frac 12)$. Then there exists some $\ve$-independent constant $C>0$ such that
    \begin{equation}\label{push}
    \frac 2{\ve}\int_\Omega  \(A[f^\ve] -A[h_{n_b}]\) \ud x+\alpha\|\mathcal{T}_\ve Q[\bar{f}^\ve]\|^2_{L^2(\RR^d)}
\leq \alpha\|\mathcal{T}_\ve Q[h_{n_b}]\|^2_{L^2(\RR^d)}+O(\sqrt{\ve})\leq C.
  \end{equation}
  \end{Lemma}

 \begin{proof}
 Since $f^\ve$ is a global minimizer and $h_{n_b}|_{\O}$ belongs to the class $\mathscr{A}$ defined in \eqref{admissible}, we have $$A_\ve[f^\ve]\leq A_\ve [h_{n_b}].$$ Thus it follows from \eqref{variation3} and \eqref{Intro-energy} that
  \begin{equation}\label{variation4}
 \begin{split}
 &\int_\O A[f^\ve](x)\ud x +\alpha\int_{\O}Q[f^\ve](x):(Q[f^\ve](x)- Q[f^\ve]*_\O g_\ve(x))\ud x\\
= &\int_{\O\times\BS}f^\ve\log f^\ve\ud x\ud m+\alpha\(\tfrac 23|\O|-\int_\O|Q[f^\ve](x)|^2\ud x\) \\ & +\alpha\int_{\O}Q[f^\ve](x):(Q[f^\ve](x)- Q[f^\ve]*_\O g_\ve(x))\ud x\\
\leq  & \int_{\O\times\BS}h_{n_b}\ln h_{n_b}\ud x\ud m+\alpha\(\tfrac 23|\O|-\int_\O|Q[h_{n_b}](x)|^2\ud x\)  \\
&+\alpha\int_{\O}Q[h_{n_b}](x):(Q[h_{n_b}](x)- Q[h_{n_b}]*_\O g_\ve(x))\ud x\\
= &\int_\O A[h_{n_b}](x)\ud x+\alpha\int_{\O}Q[h_{n_b}](x):(Q[h_{n_b}](x)- Q[h_{n_b}]*_\O g_\ve(x))\ud x.
  \end{split}
\end{equation}
On the other hand, it follows from the definition of $\bar{f^\ve}$ in \eqref{extension1} that
\begin{equation*}
\begin{split}
&\int_{\RR^d\times\RR^d}Q[\bar{f^\ve}](x):Q[\bar{f^\ve}](x')g_\ve(x-x')\ud x\ud x'
-\int_{\O^c \times\O^c }Q[h_{n_b} ](x):Q[h_{n_b} ](x')g_\ve(x-x')\ud x\ud x'\\
=&\int_{\O\times\O}Q[{f^\ve}](x):Q[{f^\ve}](x')g_\ve(x-x')\ud x\ud x' +2\int_{\O^c}
 Q[\bar{f^\ve}](x):\int_\O Q[\bar{f^\ve}](x')g_\ve(x-x')\ud x'\ud x\\
=&\int_{\O}Q[{f^\ve}](x):\(Q[{f^\ve}]*_\O g_\ve\)(x)\ud x  +2\int_{\O^c  }
Q[h_{n_b}](x):\int_\O Q[\bar{f^\ve}](x')g_\ve(x-x')\ud x'\ud x.
  \end{split}
\end{equation*}
Likewise,
\begin{equation*}
  \begin{split}
    &\int_{\RR^d\times\RR^d}Q[h_{n_b} ](x):Q[h_{n_b} ](x')g_\ve(x-x')\ud x\ud x'-\int_{\O^c \times\O^c }Q[h_{n_b} ](x):Q[h_{n_b} ](x')g_\ve(x-x')\ud x\ud x'\\
    =&\int_{\O}Q[h_{n_b}](x):\(Q[h_{n_b}]*_\O g_\ve\)(x)\ud x  +2\int_{\O^c  }Q[ h_{n_b}](x):\int_\O Q[ h_{n_b}](x')g_\ve(x-x')\ud x'\ud x.
  \end{split}
\end{equation*}
Subtracting the above two identities leads to
\begin{equation}\label{extension2}
  \begin{split}
    &\int_{\RR^d\times\RR^d}Q[\bar{f^\ve}](x):Q[\bar{f^\ve}](x')g_\ve(x-x')\ud x\ud x'-\int_{\RR^d\times\RR^d}Q[h_{n_b} ](x):Q[h_{n_b} ](x')g_\ve(x-x')\ud x\ud x'\\
    =&\int_{\O}Q[{f^\ve}](x):\(Q[{f^\ve}]*_\O g_\ve\)(x)\ud x -\int_{\O}Q[ h_{n_b} ](x):\(Q[ h_{n_b} ]*_\O g_\ve\)(x)\ud x  \\
    & +2\int_{\O^c  }Q[ h_{n_b}  ](x):\int_\O \(Q[ f^\ve ](x')-Q[h_{n_b} ](x')\)g_\ve(x-x')\ud x'\ud x\\
    =&\int_{\O}Q[{f^\ve}](x):\(Q[{f^\ve}]*_\O g_\ve\)(x)\ud x -\int_{\O}Q[ h_{n_b} ](x):\(Q[ h_{n_b} ]*_\O g_\ve\)(x)\ud x  \\
    & +2\int_{\O^c  }Q[ h_{n_b}  ](x):\int_{\O_\delta} \(Q[ f^\ve ](x')-Q[h_{n_b} ](x')\)g_\ve(x-x')\ud x'\ud x.
  \end{split}
\end{equation}
In the last step, we used the \eqref{admissible}. Applying  \eqref{decay-g}  again, the last integral in \eqref{extension2} can be estimated by
\begin{equation}\label{extension3}
  \begin{split}
    &\left|\int_{\O^c  }Q[ h_{n_b} ](x):\int_{\O_\delta} \(Q[ f^\ve ](x')-Q[h_{n_b} ](x')\)g_\ve(x-x')\ud x'\ud x\right|\\
    \leq & C e^{-\frac\delta{\sqrt\ve}}\int_{\O^c  }\left|Q[h_{n_b} ](x)\right|\ud x\int_{\O_\delta}| Q[ f^\ve ](x')-Q[h_{n_b}  ](x')|\ud x'\\
     \leq & C \ve^{\frac 32}\int_{\RR^d}|Q[h_{n_b} ](x)|\ud x\leq C\ve^{\frac 32}.
      \end{split}
\end{equation}
In the last step of above estimate, we employed Lemma \ref{lemma1} together with \eqref{bound4} to show that
\begin{equation*}
  \int_{\RR^d}|Q[h_{n_b} ](x)|\ud x\leq C<\infty.
\end{equation*}
Combining \eqref{extension2} and \eqref{extension3}
\begin{equation}\label{extension4}
  \begin{split}
    &\int_{\RR^d}Q[\bar{f^\ve}](x):(Q[\bar{f^\ve}]* g_\ve)(x)\ud x -\int_{\RR^d}Q[h_{n_b}](x):(Q[h_{n_b} ]* g_\ve)(x)\ud x \\
    =&\int_{\O}Q[{f^\ve}](x):\(Q[f^\ve]*_\O g_\ve\)(x)\ud x -\int_{\O}Q[ h_{n_b}  ](x):\(Q[h_{n_b}  ]*_\O g_\ve\)(x)\ud x  +C \ve^{\frac 32}.
  \end{split}
\end{equation}
Substituting \eqref{extension4} into \eqref{variation4} and then employing \eqref{extension1} as well as  the symmetric property of convolution $*$, we get
\begin{equation}\label{eq:1.22}
  \begin{split}
    &\frac 1\ve\int_\O(A[f^\ve]-A[h_{n_b} ])(x)\ud x+\frac{\alpha}{\ve}\int_{\RR^d}
    Q[\bar{f^\ve}](x):\Big(Q[\bar{f^\ve}](x)- (Q[\bar{f^\ve}]*g_\ve)(x)\Big)\ud x\\
    \leq &\frac{\alpha}{\ve}\int_{\RR^d}Q[h_{n_b} ](x):\Big(Q[h_{n_b} ](x)-(Q[h_{n_b} ]*g_\ve)(x)\Big)\ud x+C \ve^{\frac 12}.
  \end{split}
\end{equation}
In view of \eqref{squrtroot} and \eqref{multiplier}, this implies the first inequality of \eqref{push}. To prove the second part, we shall estimate the right hand side of \eqref{eq:1.22}.
For the sake of simplicity, we shall denote $Q[h_{n_b} ]$ by $v$ in the remaining part of the proof:
   \begin{equation*}
     \begin{split}
    & 2\int_{\RR^d}v(x):\(v(x)-(v *g_\ve)(x)\)\ud x\\
 =   &    \int_{\RR^d\times\RR^d}\left|v(x)-v(y)\right|^2
g_\ve\(x-y\)\ud x\ud y\\
=&\int_{\RR^d\times\RR^d}\left|v(y+z)-v(y)\right|^2
g_\ve\(z\)\ud z\ud y\\
\leq &\int_{\RR^d\times\RR^d}|z|^2\int_0^1\left|\nabla v(yt+(1-t)(y+z))\right|^2\ud t
g_\ve\(z\)\ud z\ud y\\
=&\ve\int_0^1\int_{\RR^d}|\tfrac z{\sqrt{\ve}}|^2g_\ve(z)\(\int_{\RR^d}\left|\nabla v(y+(1-t)z)\right|^2
dy\)\ud z\ud t\\
=&\ve\int_0^1\int_{\RR^d}|\tfrac z{\sqrt{\ve}}|^2g_\ve(z)\|\nabla  v \|^2_{L^2(\RR^d)}\ud z\ud t\\
=& \ve \|\nabla v \|^2_{L^2(\RR^d)}\int_{\RR^d}|x|^2g(x)\ud x
\leq  C\|\nabla n_b\|_{L^2(\RR^d)}^2\leq C.
     \end{split}
   \end{equation*}
   In the last step, we employed Lemma \ref{lemma1} and \eqref{bound4} successively.
\end{proof}
Now we are in position to prove the second statement of Theorem \ref{thmglo}.
\begin{proof}[Complete the proof of Theorem \ref{thmglo}]
The existence of solution $f^\ve$ to \eqref{globalmini2} is proved in Theorem \ref{existence}. Moreover,  it follows from Lemma \ref{eq:1.03} that $f^\ve$ satisfies \eqref{push}, which is equivalent to \eqref{push2}, according to \eqref{defteps}. So the hypothesises for applying Proposition \ref{compactness1} is fulfilled for the global minimizers $f^\ve$:
\begin{equation*}
\bar{f^\ve}\rightharpoonup \bar{f}~\text{ weakly  in}~L_{loc}^1(\RR^d\times\BS),
\end{equation*}
where  $\bar{f}(x,m)$ is given by
\begin{equation*}
\bar{f}(x,m)=\left\{
\begin{array}{rl}
\frac {e^{\eta(m\cdot n(x))^2}}{\int_\BS e^{\eta(m\cdot n(x))^2}dm}&~\text{for}~ x\in \O,\\
h_{n_b}(x,m)&~\text{for}~ x\in \RR^d\backslash\O
\end{array}
\right.
\end{equation*}
  for some $n(x)\in H^1(\O;\BS)$. Likewise, we can show, as in the second part of the proof of  Theorem \ref{thmcri}, that
   \begin{equation}\label{eq:1.04}
     n(x)|_{\p\O}=\pm n_b.
   \end{equation}
    It remains to show that $n(x)$ is a minimizing harmonic map provided that $n_b|_\O$ is a harmonic map with boundary condition $n_b|_{\p\O}$. It follows from  \eqref{push} that
       \begin{equation}\label{eq:1.23}
    \frac 2{\ve}\int_\Omega  \(A[f^\ve] -A[h_{n_b}]\) \ud x+\alpha\|\mathcal{T}_\ve Q[\bar{f}^\ve]\|^2_{L^2(\RR^d)}
\leq \alpha\|\mathcal{T}_\ve Q[h_{n_b}]\|^2_{L^2(\RR^d)}+O(\sqrt{\ve}).
  \end{equation}
On the other hand, Lemma \ref{converge-T} and  Lemma \ref{variation5} state that
\begin{equation*}
\begin{split}
   \mathcal{T}_\ve Q[\bar{f^\ve}]&\rightharpoonup -i\sqrt{\tfrac \mu d} \nabla Q[f]~\text{ weakly in }L^2_{loc}(\RR^d),\\
    \mathcal{T}_\ve Q[h_{n_b}]&\rightharpoonup -i\sqrt{\tfrac \mu d} \nabla Q[h_{n_b}]~\text{strongly in }L^2(\RR^d).
\end{split}
\end{equation*}
These enable us to pass to the limit in \eqref{eq:1.23} using  lower semi-continuity
\begin{equation*}
\|\nabla Q[\bar{f}]\|^2_{L^2(\RR^d)}\leq\|\nabla Q[h_{n_b}]\|^2_{L^2(\RR^d)},
\end{equation*}
which is equivalent to
\begin{equation*}
\int_{\O}|\nabla Q[\bar{f}]|^2(x)\ud x+\int_{\RR^d\backslash\O}|\nabla Q[h_{n_b}]|^2(x)\ud x
\leq \int_{\O}|\nabla Q[h_{n_b}]|^2(x)\ud x+\int_{\RR^d\backslash\O}|\nabla Q[h_{n_b}]|^2(x)\ud x.
\end{equation*}
Owning to \eqref{uniaxialtensor}, we obtain
$$\|\nabla n(x)\|_{L^2(\O)}\leq \|\nabla n_b(x)\|_{L^2(\O)}.$$
This combined with \eqref{eq:1.04} implies that $n(x)$ is a minimizing harmonic map.
\end{proof}

\section*{Acknowledgments}
The authors would like to thank Professor Zhifei Zhang for helpful discussions. W. Wang is partly supported by NSF of China under Grant 11501502.

\end{document}